\documentclass[11pt, a4paper]{article}
\usepackage{fix-cm}

\usepackage{graphics}
\usepackage{graphicx}
\usepackage{epsfig}
\usepackage{amsmath}
\usepackage{graphics}
\usepackage{graphicx}
\usepackage{epstopdf}
\usepackage{multirow}
\usepackage{booktabs}
\usepackage{subfigure}
\usepackage{ifpdf}
\usepackage{color}
\usepackage{amssymb}
\usepackage{amsthm}
\usepackage{amsmath}
\usepackage{mathrsfs}
\usepackage{bm}
\usepackage[ruled,vlined,linesnumbered,noend]{algorithm2e}
\SetKwInput{KwIn}{Input}
\SetKwInput{KwOut}{Output}

\usepackage[numbers,sort&compress]{natbib}% Citation support using natbib.sty
\bibpunct[, ]{[}{]}{,}{n}{,}{,}% Citation support using natbib.sty
\makeatletter% @ becomes a letter
\def\NAT@def@citea{\def\@citea{\NAT@separator}}% Suppress spaces between citations using natbib.sty
\makeatother% @ becomes a symbol again

\usepackage{xparse}
\makeatletter
\newcommand{\sn@authors}{}
\newcommand{\sn@affils}{}
\newcommand{\sn@abstract}{}
\newcommand{\sn@keywords}{}
\newcounter{sn@authorcount}
\RenewDocumentCommand{\author}{s o m}{%
  \stepcounter{sn@authorcount}%
  \ifnum\value{sn@authorcount}>1
    \g@addto@macro\sn@authors{, }
  \fi
  \g@addto@macro\sn@authors{#3}%
  \IfNoValueTF{#2}{%
    \IfBooleanT{#1}{\g@addto@macro\sn@authors{\textsuperscript{*}}}%
  }{%
    \IfBooleanTF{#1}{%
      \g@addto@macro\sn@authors{\textsuperscript{#2,*}}%
    }{%
      \g@addto@macro\sn@authors{\textsuperscript{#2}}%
    }%
  }%
}
\NewDocumentCommand{\affil}{s o m}{%
  \ifx\sn@affils\@empty
    \IfNoValueTF{#2}{%
      \gdef\sn@affils{#3}%
    }{%
      \gdef\sn@affils{\textsuperscript{#2} #3}%
    }%
  \else
    \IfNoValueTF{#2}{%
      \g@addto@macro\sn@affils{\par#3}%
    }{%
      \g@addto@macro\sn@affils{\par\textsuperscript{#2} #3}%
    }%
  \fi
}
\RenewDocumentCommand{\abstract}{m}{\gdef\sn@abstract{#1}}
\NewDocumentCommand{\keywords}{m}{\gdef\sn@keywords{#1}}
\newcommand{\email}[1]{\g@addto@macro\sn@authors{\par\texttt{#1}}}
\newcommand{\fnm}[1]{#1}
\newcommand{\sur}[1]{#1}
\newcommand{\orgdiv}[1]{#1}
\newcommand{\orgname}[1]{#1}
\newcommand{\orgaddress}[1]{#1}
\newcommand{\city}[1]{#1}
\newcommand{\postcode}[1]{#1}
\newcommand{\country}[1]{#1}
\newcommand{\street}[1]{#1}
\newcommand{\state}[1]{#1}
\renewcommand{\maketitle}{%
  \begin{center}
    {\LARGE \@title\par}
    \vspace{1em}
    {\large \sn@authors\par}
    \vspace{0.8em}
    {\small \sn@affils\par}
  \end{center}
  \ifx\sn@abstract\@empty\else
    \noindent\textbf{Abstract.} \sn@abstract\par\vspace{1em}
  \fi
  \ifx\sn@keywords\@empty\else
    \noindent\textbf{Keywords:} \sn@keywords\par\vspace{1em}
  \fi
}
\makeatother

\newcommand\numberthis{\addtocounter{equation}{1}\tag{\theequation}}

\theoremstyle{plain}% Theorem-like structures provided by amsthm.sty
\newtheorem{theorem}{Theorem}[section]
\newtheorem{lemma}[theorem]{Lemma}
\newtheorem{corollary}[theorem]{Corollary}
\newtheorem{proposition}[theorem]{Proposition}
\newtheorem{problem}[theorem]{Problem}

\theoremstyle{definition}
\newtheorem{definition}[theorem]{Definition}
\newtheorem{example}[theorem]{Example}

\theoremstyle{remark}

\newcommand{\R}{\mathbb{R}}
\newcommand{\C}{\mathbb{C}}

\newcommand{\mX}{\mathscr{X}}
\newcommand{\mY}{\mathscr{Y}}
\newcommand{\mW}{\mathscr{W}}
\newcommand{\mH}{\mathscr{H}}
\newcommand{\mA}{\mathscr{A}}
\newcommand{\mB}{\mathscr{B}}
\newcommand{\mC}{\mathscr{C}}
\newcommand{\mD}{\mathscr{D}}
\newcommand{\mE}{\mathscr{E}}

\newcommand{\mG}{\mathscr{G}}
\newcommand{\mI}{\mathscr{I}}

\newcommand{\mM}{\mathscr{M}}

\newcommand{\mO}{\mathscr{O}}

\newcommand{\mR}{\mathscr{R}}
\newcommand{\mU}{\mathscr{U}}
\newcommand{\mS}{\mathscr{S}}
\newcommand{\mV}{\mathscr{V}}
\newcommand{\mT}{\mathscr{T}}
\newcommand{\mZ}{\mathscr{Z}}

\DeclareMathOperator{\bcirc}{bcirc}
\DeclareMathOperator{\unfold}{unfold}
\DeclareMathOperator{\fold}{fold}
\DeclareMathOperator{\diag}{diag}
\DeclareMathOperator{\spec}{spec}
\DeclareMathOperator{\rank}{rank}
\DeclareMathOperator{\prank}{prank}
\DeclareMathOperator{\Tprank}{Tprank}

\DeclareMathOperator*{\argmin}{argmin}

\begin{document}

\title{Low T-Phase Rank Approximation of Third Order Tensors}

\author[1]{\fnm{Taehyeong} \sur{Kim}}
\author[1,2]{\fnm{Hayoung} \sur{Choi}}
\author*[3]{\fnm{Yimin} \sur{Wei}}\email{ymwei@fudan.edu.cn}

\affil*[1]{\orgdiv{Nonlinear Dynamics and Mathematical Application Center}, \orgname{Kyungpook National University}, \orgaddress{ \city{Daegu}, \postcode{41566}, \country{Republic of Korea}}}

\affil[2]{\orgdiv{Department of Mathematics}, \orgname{Kyungpook National University}, \orgaddress{ \city{Daegu}, \postcode{41566}, \country{Republic of Korea}}}

\affil[3]{\orgdiv{School of Mathematical Sciences}, \orgname{Fudan University}, \orgaddress{\street{220, Handan Road}, \city{Shanghai}, \postcode{200433}, \state{Shanghai}, \country{China}}}

\abstract{
We study low T-phase-rank approximation of sectorial third-order tensors $\mA\in\C^{n\times n\times p}$ under the tensor T-product.
We introduce canonical T-phases and T-phase rank, and formulate the approximation task as minimizing a symmetric gauge of the canonical phase vector under a T-phase-rank constraint.
Our main tool is a tensor phase-majorization inequality for the geometric mean, obtained by lifting the matrix inequality through the block-circulant representation.
In the positive-imaginary regime, this yields an exact optimal-value formula and an explicit optimal half-phase truncation family.
We further establish tensor counterparts of classical matrix phase inequalities and derive a tensor small phase theorem for MIMO linear time-invariant systems.
}

\keywords{Low phase rank approximation; T-product; Tensor phase; Geometric mean}

\maketitle

\section{Introduction}

Third-order tensors equipped with the T-product form a matrix-like algebra: using the block-circulant embedding together with the discrete Fourier transform (DFT), the T-product reduces to ordinary matrix multiplication in the Fourier domain~\citep{kilmer2008third,kilmer2011factorization,kilmer2013third}.
This perspective has enabled many tensor counterparts of familiar matrix constructions, including T-SVD and related fast or randomized algorithms~\citep{Wei2024singular,ding2023randomized,che2022fast,miao2020generalized}, spectral notions and perturbation analysis based on T-eigenvalues~\citep{miao2021t,mo2024perturbation}, and factorization tools such as tensor LU/QR and CUR-type decompositions~\citep{zhu2022tensor,chen2022tensor,chen2024coseparable}.
Operator means have also been developed for T-positive definite tensors, in particular, a T-product geometric mean~\citep{ju2024geometric}.
Low-rank tensor approximation remains a basic ingredient for compression and denoising and appears broadly in multiway data analysis~\citep{song2019relative,friedland2015low,nie2017low,deng2023solving,rozada2024tensor,jiang2020nonnegative,lebeau2024random}.

For sectorial matrices, recent work has emphasized a complementary phase-oriented viewpoint.
One associates a canonical phase vector and defines the phase rank as the number of nonzero canonical phases; phase majorization inequalities, especially those involving the matrix geometric mean, then yield stability statements such as small phase theorems~\citep{zhao2022low}.
These ideas have been further exploited in control for MIMO LTI systems, where phase constraints complement gain-based criteria and lead to robust stability tests and synthesis guidelines~\citep{chen2019phase,chen2024phase2,Mao2022,SRAZHIDINOV2023,wang2020phases,WANG2023on,Chen2024phase1,Chen2023,WANG2024,Wang2024first}.

The goal of this paper is to develop an analogous phase framework for third-order tensors under the T-product.
We adopt the notion of sectorial tensors introduced in~\citep{ding2025tensor}: a tensor is sectorial when its block-circulant embedding is a sectorial matrix.
For a sectorial tensor, we define the canonical T-phase vector and the associated T-phase rank.
We then formulate a low T-phase-rank approximation problem: among tensors with prescribed T-phase rank, we seek an approximant that minimizes a symmetric gauge of the canonical phase vector.
Our main technical ingredient is a phase-majorization inequality for the tensor geometric mean, obtained by lifting a known matrix inequality to the block-circulant setting~\citep{Chen2024phase1,Chen2023,chen2019phase,chen2024phase2}.
Finally, we apply the resulting phase calculus to tensor-form MIMO LTI systems and derive a tensor version of the small phase theorem.

Relative to the matrix results of~\citep{zhao2022low}, our contributions and scope are:
\begin{itemize}
\item[(i)] We establish a tensor phase-majorization inequality for the T-product geometric mean via block-circulant lifting.
\item[(ii)] We construct a globally sorted half-phase truncation family and derive the exact optimal objective value for low-$\Tprank$ approximation in the positive-imaginary setting.
\item[(iii)] We prove tensor analogues of several phase inequalities and connect the phase framework to tensor MIMO stability analysis.
\item[(iv)] The exactness proof is obtained by reducing the tensor problem to the matrix formulation in~\citep[Theorem~2]{zhao2022low} via the block-circulant representation.
\end{itemize}

While the T-product literature has largely focused on magnitude-driven quantities such as T-singular values and tensor ranks~\citep{kilmer2011factorization,kilmer2013third,Wei2024singular}, phase information - central in the analysis of sectorial matrices - has not been systematically developed for tensors.
The canonical T-phases and T-phase rank introduced here provide a phase-sensitive approximation theory that complements existing magnitude-based approaches.

The paper is structured as follows.
Section~2 collects notation for the T-product framework and recalls the basic notions needed for
sectorial tensors and their canonical phases.
Section~3 defines T-phase gauge and establishes the phase inequalities that will be used later,
with an emphasis on a majorization bound for the tensor geometric mean.
Section~4 formulates a low T-phase-rank approximation problem and derives a constructive
half-phase truncation procedure with an exact objective formula (in the positive-imaginary regime).
Section~5 applies the resulting phase calculus to tensor-form MIMO LTI systems and derives a tensor
small phase theorem with stability consequences.

\section{Preliminaries}\label{sec:preliminaries}

For $z\in\C\setminus\{0\}$ we use the principal polar representation
$z=|z|e^{j\angle z}$ with $\angle z\in(-\pi,\pi]$ and $j=\sqrt{-1}$.
We use weak/strong majorization notation as in \citep{marshall1979inequalities}.
For a third-order tensor $\mA=[a_{ijk}]\in\C^{m\times n\times p}$, the $k$th frontal slice is denoted by
$\bm A^{(k)}\in\C^{m\times n}$ for $k\in[p]:=\{1,2,\ldots,p\}$.
The Frobenius inner product is $\langle \mA,\mB\rangle_F:=\sum_{i,j,k}\overline{a}_{ijk}b_{ijk}$ and
$\|\mA\|_F:=\sqrt{\langle \mA,\mA\rangle_F}$.

We adopt the T-product framework of \citep{kilmer2008third,kilmer2011factorization}.
Let $\bcirc(\mA)\in\C^{mp\times np}$ be the block-circulant embedding of $\mA$ and
$\unfold(\mA)\in\C^{mp\times n}$ be the unfolding obtained by stacking the frontal slices; let $\fold$ be its inverse.
For conforming tensors $\mA\in\C^{m\times n\times p}$ and $\mB\in\C^{n\times s\times p}$, the T-product is
\[
\mA*\mB := \fold\big(\bcirc(\mA)\unfold(\mB)\big).
\]
The identity tensor $\mI_{np}\in\R^{n\times n\times p}$ satisfies $\bcirc(\mI_{np})=\bm I_{np}$.
For $\mA\in\C^{m\times n\times p}$, the conjugate transpose $\mA^H$ is defined by $\bcirc(\mA^H)=\bcirc(\mA)^H$.
A frontal-square tensor $\mA\in\C^{n\times n\times p}$ is \emph{T-Hermitian} if $\mA=\mA^H$; we denote the set by $\mathbb{H}^{n\times n\times p}$.
A tensor $\mA\in\mathbb{H}^{n\times n\times p}$ is \emph{T-positive (semi-)definite} if $\langle \mX,\mA*\mX\rangle_F>0$ ($\ge 0$) for all nonzero $\mX\in\C^{n\times 1\times p}$; we denote the corresponding cones by $\mathbb{H}^{n\times n\times p}_{++}$ and $\mathbb{H}^{n\times n\times p}_{+}$ \citep{zheng2021t}.

A frontal-square tensor $\mA\in\C^{n\times n\times p}$ is said to be \emph{T-invertible} if there exists a tensor $\mA^{-1}\in\C^{n\times n\times p}$ such that $\mA*\mA^{-1}=\mA^{-1}*\mA=\mI_{np}$.
In this case, we call $\mA^{-1}$ the \emph{T-inverse} of $\mA$.
Equivalently, the block-circulant embedding $\bcirc(\mA)$ is nonsingular and
\begin{equation}\label{eq:Tinverse}
\bcirc(\mA^{-1})=\bcirc(\mA)^{-1}.
\end{equation}

Following \citep{zheng2021t}, we define the \emph{T-eigenvalues} of $\mA$ as the eigenvalues of $\bcirc(\mA)$.
That is, a scalar $\lambda\in\C$ is a T-eigenvalue of $\mA\in\C^{n\times n\times p}$ if $\lambda$ is an eigenvalue of $\bcirc(\mA)$, and we denote the multiset of all T-eigenvalues by $\spec(\mA)$.

\begin{lemma}[Fourier block diagonalization; see \citep{kilmer2011factorization}]\label{LemmaFourierBlockDiag}
Let $\mA\in\C^{m\times n\times p}$ and let $\bm F_p$ be the $p\times p$ unitary DFT matrix.
Then
\begin{equation}\label{eqFourierBlockDiagMN}
\bcirc(\mA) = (\bm F_p^H\otimes \bm I_m)\,\diag(\bm A_1,\ldots,\bm A_p)\,(\bm F_p\otimes \bm I_n),
\end{equation}
where $\bm A_1,\ldots,\bm A_p$ are the frontal slices of the FFT of $\mA$ along the third mode.
\end{lemma}

Following \citep{ding2025tensor}, the numerical range of $\mA\in\C^{n\times n\times p}$ is
\begin{equation}\label{def:numerical_range_tensor}
W(\mA):=\big\{\langle \mX,\mA*\mX\rangle_F \ | \ \mX\in\C^{n\times 1\times p},\ \|\mX\|_F=1\big\}.
\end{equation}
Moreover, $W(\mA)=W(\bcirc(\mA))$ and hence $W(\mA)$ is convex \citep{ding2025tensor}.
We call $\mA$ \emph{sectorial} if $0\notin W(\mA)$.

\begin{theorem}[Sectorial tensor decomposition; \citep{ding2025tensor}]\label{thm:sectorial_decomposition}
Let $\mA\in\C^{n\times n\times p}$ be sectorial. Then there exist a nonsingular tensor $\mT$ and a diagonal unitary tensor $\mD$ such that
\begin{equation}\label{eq:sectorial_decomposition}
\mA=\mT^H*\mD*\mT.
\end{equation}
\end{theorem}

\begin{definition}[Canonical T-phases]\label{def:canonical_phases}
Let $\mA\in\C^{n\times n\times p}$ be sectorial.
Then $\bcirc(\mA)$ is a sectorial matrix and admits a sectorial decomposition
$\bcirc(\mA)=\bm T^H \bm D \bm T$,
where $\bm T$ is nonsingular and $\bm D$ is diagonal unitary.
Write the diagonal entries of $\bm D$ as $e^{j\phi_k(\mA)}$ with $\phi_k(\mA)\in(-\pi,\pi]$.
We order them so that
\[
\phi_1(\mA)\ge\phi_2(\mA)\ge\cdots\ge\phi_{np}(\mA),
\qquad \phi_1(\mA)-\phi_{np}(\mA)<\pi.
\]
The vector $\phi(\mA):=(\phi_1(\mA),\ldots,\phi_{np}(\mA))$ is called the \emph{canonical T-phase vector}.
We also set $\overline{\phi}(\mA):=\phi_1(\mA)$ and $\underline{\phi}(\mA):=\phi_{np}(\mA)$.
\end{definition}

Fix $\alpha,\beta\in\R$ with $0<\beta-\alpha\le\pi$ and $\frac{\alpha+\beta}{2}\in(-\pi,\pi]$.
A sectorial tensor $\mA\in\C^{n\times n\times p}$ is said to lie in the sector $[\alpha,\beta]$ if
$[\underline{\phi}(\mA),\overline{\phi}(\mA)]\subseteq[\alpha,\beta]$; we denote this by $\mA\in\mathcal{C}[\alpha,\beta]$.
We also write $\mathcal{C}(\alpha,\beta)$ for the corresponding open sector.
We call $\mA$ \emph{quasi-sectorial} if $\mA\in\mathcal{C}[\alpha,\beta]$ for some $\beta-\alpha<\pi$, and
\emph{semi-sectorial} if $\mA\in\mathcal{C}[\alpha,\beta]$ for some $\beta-\alpha\le\pi$.
In particular, $\mA$ is \emph{positive-real} (accretive) when $\mA\in\mathcal{C}(-\frac{\pi}{2},\frac{\pi}{2})$, and \emph{negative-imaginary} when $\mA\in\mathcal{C}(-\pi,0]$.

\section{T-phase gauges and T-phase inequalities}
Let $\mA \in \C^{n\times n\times p}$ be sectorial. We denote its T-eigenvalues and ordered T-singular values by $\lambda(\mA)=[\lambda_k(\mA)]_{k=1}^{np}$ and $\sigma(\mA)=[\sigma_{k}(\mA)]_{k=1}^{np}$, respectively, where $\sigma_{1}(\mA)\ge\cdots\ge\sigma_{np}(\mA)$.
We write $\phi(\mA)=[\phi_k(\mA)]_{k=1}^{np}$ for the canonical T-phase vector and, for later convenience, set $\phi_{np+1}(\mA)=0$.

The tensor functional calculus in \citep{ju2024geometric} defines $\mA^{\alpha}$ (real $\alpha$) for T-Hermitian positive definite tensors.
Here, we need the principal power for a broader class, namely, strictly accretive tensors.
We therefore consider
\begin{equation*}
    \mathbb{T}^{n\times n\times p}_{+}
    = \left\{ \mA \in \mathbb{C}^{n\times n\times p} \ \middle|\ \Re\langle \mX,\mA*\mX\rangle_F > 0 \ \text{for all } \mX\neq 0 \right\}
    = \mathcal{C}\!\left(-\frac{\pi}{2},\frac{\pi}{2}\right).
\end{equation*}
In particular, $\mathbb{T}_{+}^{n\times n\times p}$ implies $\Re(\lambda_i(\mA))>0$ for all T-eigenvalues.
Whenever principal powers or the Drury-type geometric mean are used, we assume tensors lie in $\mathbb{T}_{+}^{n\times n\times p}$.

For $\mA \in \mathbb{T}^{n\times n\times p}_{+}$, by Lemma~\ref{LemmaFourierBlockDiag}, we define $\mA^{\alpha}$ as
\begin{align*}
    \bcirc(\mA^{\alpha}) 
    =& (\bm{F}_p^H \otimes \bm{I}_n)\cdot \diag(\bm{A}_1,\bm{A}_2,\ldots,\bm{A}_p)^{\alpha}\cdot (\bm{F}_p \otimes \bm{I}_n)\\
    =& (\bm{F}_p^H \otimes \bm{I}_n)\cdot \diag(\bm{A}_1^{\alpha},\bm{A}_2^{\alpha}, \ldots,\bm{A}_p^{\alpha})\cdot (\bm{F}_p \otimes \bm{I}_n),
\end{align*}
where each $\bm{A}_{i}^{\alpha}$ is calculated using the Dunford--Taylor integral~\citep{drury2015principal}:
\begin{equation}\label{eq:matrix_t_power}
    \bm{A}_{i}^{\alpha} = \frac{1}{2\pi i} \oint_{\Gamma} (z\bm{I}_{n} - \bm{A}_{i})^{-1} z^{\alpha} \, dz.
\end{equation}
Here, $\Gamma$ is a closed contour that winds around each eigenvalue of $\bm{A}_{i}$ exactly once and avoids the branch cut $(-\infty, 0]$.
The principal power is defined by
\[
z^\alpha := r^\alpha e^{i\alpha\theta}
\quad\text{for }z=re^{i\theta}\text{ with }r>0,\;-\pi<\theta<\pi,
\]
where $\theta=\mathrm{Arg}(z)$ denotes the principal argument.
Moreover, since $\mA^{\alpha}$ is well defined, we have $\mA^{-\alpha}=(\mA^{\alpha})^{-1}$ when $\mA^{\alpha}$ is nonsingular.

Following Drury's extension of the matrix geometric mean to accretive matrices~\citep{drury2015principal}, we define a geometric mean on $\mathbb{T}^{n\times n\times p}_{+}$ by an integral formula.
Specifically, for $\mA,\mB \in \mathbb{T}^{n\times n\times p}_{+}$ we define $\mX=\mA \# \mB$ through
\begin{equation}\label{eq:extend_geometric_mean}
    (\mA \# \mB)^{-1} = \frac{2}{\pi} \int_{0}^{\infty} \left(t\mA + t^{-1}\mB \right)^{-1} \frac{dt}{t}.
\end{equation}
This definition is symmetric in $\mA$ and $\mB$.

\begin{proposition}\label{PropBcircGeometricMean}
For $\mA,\mB\in\mathbb{T}^{n\times n\times p}_{+}$, the block-circulant embedding commutes with the geometric mean:
\[
\bcirc(\mA\#\mB)=\bcirc(\mA)\#\bcirc(\mB).
\]
\end{proposition}
\begin{proof}
Apply $\bcirc(\cdot)$ to \eqref{eq:extend_geometric_mean}.
By the linearity of $\bcirc$ and \eqref{eq:Tinverse},
\[
\bcirc\big((\mA\#\mB)^{-1}\big)
=\frac{2}{\pi}\int_{0}^{\infty}\big(t\,\bcirc(\mA)+t^{-1}\,\bcirc(\mB)\big)^{-1}\frac{dt}{t}.
\]
The right-hand side is exactly the integral representation of $(\bcirc(\mA)\#\bcirc(\mB))^{-1}$. Hence,
$\bcirc\big((\mA\#\mB)^{-1}\big)=(\bcirc(\mA)\#\bcirc(\mB))^{-1}$.
Taking inverses gives the claim.
\end{proof}

Moreover, each $\mA \in \mathbb{T}_{+}^{n\times n\times p}$ admits a unique principal square root $\mA^{\frac{1}{2}}\in\mathbb{T}_{+}^{n\times n\times p}$ satisfying $\mA = \mA^{\frac{1}{2}} * \mA^{\frac{1}{2}}$; see \citep[Theorem 1.29]{higham2008functions}.
With this notation, the tensor geometric mean admits the standard closed-form expression:
\begin{theorem}
Let $\mA,\mB \in \mathbb{T}^{n\times n\times p}_{+}$, then
$$
\mX = \mA^{\frac{1}{2}} * \left(\mA^{-\frac{1}{2}} * \mB * \mA^{-\frac{1}{2}}\right)^{\frac{1}{2}} * \mA^{\frac{1}{2}}
= \mB^{\frac{1}{2}} * \left( \mB^{-\frac{1}{2}} * \mA * \mB^{-\frac{1}{2}} \right)^{\frac{1}{2}} * \mB^{\frac{1}{2}}.
$$
\end{theorem}

\begin{proof}
Let $\lambda$ be an eigenvalue of $\mA^{-\frac{1}{2}} * \mB * \mA^{-\frac{1}{2}}$, then there exists a vector $\bm{v}\neq 0$ such that
$$
\bcirc \left(\mA^{-\frac{1}{2}} * \mB * \mA^{-\frac{1}{2}} \right) \bm{v} = \lambda \bm{v}.
$$
Define $\bm{\eta} := \bcirc(\mA)^{-\frac{1}{2}} \bm{v} \neq 0$. Then $\bcirc(\mB)\bm{\eta} = \lambda \bcirc(\mA)\bm{\eta}$.
Left-multiplying by $\bm{\eta}^*$ gives $\bm{\eta}^* \bcirc(\mB) \bm{\eta} = \lambda \bm{\eta}^* \bcirc(\mA) \bm{\eta}$.
Let $a = \bm{\eta}^{*} \bcirc(\mA) \bm{\eta}$ and $b = \bm{\eta}^{*} \bcirc(\mB) \bm{\eta}$.
Since $\mA, \mB \in \mathbb{T}^{n\times n\times p}_{+}$, both $a$ and $b$ lie in the open right half-plane of $\mathbb{C}$, so $\lambda = b/a$.
Since $\angle a,\angle b\in(-\pi/2, \pi/2)$, $\lambda$ satisfies $\angle \lambda = \angle b - \angle a \in (-\pi, \pi)$.
This implies that $\lambda$ does not lie on the negative real axis $(-\infty, 0]$, and thus the principal square root is well-defined.
Therefore,
\begin{align*}
\mA^{\frac{1}{2}} * \mX^{-1} * \mA^{\frac{1}{2}} =& \frac{2}{\pi} \int_{0}^{\infty} \mA^{\frac{1}{2}}  *\left(\mB + t^2\mA \right)^{-1} * \mA^{\frac{1}{2}} \, dt \numberthis \label{eq:3.3} \\
=& \frac{2}{\pi} \int_{0}^{\infty} \left(\mA^{-\frac{1}{2}} * \mB * \mA^{-\frac{1}{2}} + t^2\mI_{np} \right)^{-1}  \, dt\\
=& \left(\mA^{-\frac{1}{2}} * \mB * \mA^{-\frac{1}{2}} \right)^{-\frac{1}{2}}.
\end{align*}
Taking inverses in \eqref{eq:3.3} yields
\[
\mX=\mA^{\frac{1}{2}} * \left(\mA^{-\frac{1}{2}} * \mB * \mA^{-\frac{1}{2}}\right)^{\frac{1}{2}} * \mA^{\frac{1}{2}}.
\]
Since \eqref{eq:extend_geometric_mean} is symmetric in $(\mA,\mB)$, the same argument with $\mA$ and $\mB$ interchanged gives
\[
\mX=\mB^{\frac{1}{2}} * \left(\mB^{-\frac{1}{2}} * \mA * \mB^{-\frac{1}{2}}\right)^{\frac{1}{2}} * \mB^{\frac{1}{2}}.
\]
\end{proof}
As an immediate consequence, we obtain:
\begin{theorem}
     Let $\mA,\mB \in \mathbb{T}^{n\times n\times p}_{+}$ and $\mX = \mA \# \mB$. Then $\mX * \mA^{-1} * \mX = \mB$.
\end{theorem}
\begin{proof}
By squaring \eqref{eq:3.3} as:
\begin{equation}\label{eq:3.4}
    \mA^{\frac{1}{2}} * \mX^{-1} * \mA * \mX^{-1} * \mA^{\frac{1}{2}} = \mA^{\frac{1}{2}} * \mB^{-1} * \mA^{\frac{1}{2}}.
\end{equation}
Premultiplying and postmultiplying by $\mA^{-1/2}$ gives
\[
\mX^{-1} * \mA * \mX^{-1} = \mB^{-1}.
\]
Taking inverses yields $\mX * \mA^{-1} * \mX = \mB$.
\end{proof}
The next theorem establishes the uniqueness of the characterization \eqref{eq:extend_geometric_mean}.
\begin{theorem}
Let $\mA, \mB, \mY \in \mathbb{T}^{n\times n\times p}_{+}$ and set $\mX:=\mA\#\mB$.
If $\mY * \mA^{-1} * \mY = \mB$, then $\mY = \mX$.
\end{theorem}
\begin{proof}
Since $\mY^{-1} * \mA * \mY^{-1} = \mB^{-1}$ and $\mY^{-1} * \mB * \mY^{-1} = \mA^{-1}$, we obtain
\begin{equation*}
\mY^{-1} * \left(t\mA + t^{-1}\mB \right) * \mY^{-1} = t\mB^{-1} + t^{-1}\mA^{-1}.
\end{equation*}
Hence,
\begin{equation*}
\mY * \left( t\mA + t^{-1}\mB \right)^{-1} * \mY = \left( t\mB^{-1} + t^{-1}\mA^{-1} \right)^{-1}.
\end{equation*}
Substituting into \eqref{eq:extend_geometric_mean} gives the following.
\begin{equation*}
\mY * \mX^{-1} * \mY = \frac{2}{\pi} \int_{0}^{\infty} \left( t\mB^{-1} + t^{-1}\mA^{-1} \right)^{-1} \frac{dt}{t} = \left( \mB^{-1} \# \mA^{-1} \right)^{-1}.
\end{equation*}
By the closed-form identity above (applied to $(\mB^{-1},\mA^{-1})$) and $(\mZ^{-1})^{1/2}=(\mZ^{1/2})^{-1}$ for principal roots,
\[
\left( \mB^{-1} \# \mA^{-1} \right)^{-1}
= \mB^{\frac{1}{2}} * \left(\mB^{-\frac{1}{2}} * \mA * \mB^{-\frac{1}{2}}\right)^{\frac{1}{2}} * \mB^{\frac{1}{2}}
= \mX.
\]
% It can be deduced as $\left(\mY * \mX^{-1}\right)^{-1} = \mI$, and since $\mY * \mX^{-1}$ can not have negative eigenvalue, $\mY * \mX^{-1} = \mI$ and $\mY = \mX$.
From $\mY * \mX^{-1} * \mY = \mX$, we have $(\mY * \mX^{-1})^2 = \mI_{np}$.
Thus, the eigenvalues of $\bcirc(\mY * \mX^{-1})$ are $\pm 1$.
Suppose $-1$ is an eigenvalue with eigenvector $\bm{v}\neq 0$, i.e., $\bcirc(\mY * \mX^{-1})\bm{v} = -\bm{v}$.
Let $\bm{\eta}:=\bcirc(\mX^{-1})\bm{v}\neq 0$. Then $\bcirc(\mY)\bm{\eta} = -\bcirc(\mX)\bm{\eta}$, hence
$\bm{\eta}^*\bcirc(\mY)\bm{\eta} = -\bm{\eta}^*\bcirc(\mX)\bm{\eta}$.
Since $\mY, \mX \in \mathbb{T}_{+}^{n\times n\times p}$, we have
$\Re(\bm{\eta}^*\bcirc(\mY)\bm{\eta})>0$ and $\Re(\bm{\eta}^*\bcirc(\mX)\bm{\eta})>0$, so
$\Re(-\bm{\eta}^*\bcirc(\mX)\bm{\eta})<0$, a contradiction.
Therefore, $-1$ is not an eigenvalue of $\mY * \mX^{-1}$.
Together with $(\mY * \mX^{-1})^2=\mI_{np}$, this gives
\[
(\mY * \mX^{-1}-\mI_{np})*(\mY * \mX^{-1}+\mI_{np})=\mO,
\]
and $\mY * \mX^{-1}+\mI_{np}$ is invertible (its eigenvalues are all $2$), hence $\mY * \mX^{-1}=\mI_{np}$ and $\mY=\mX$.
\end{proof}

\begin{proposition}\label{prop:congruence_geomean}
Let $\mA,\mB\in\mathbb{T}^{n\times n\times p}_{+}$ and let $\mT\in\C^{n\times n\times p}$ be nonsingular. Then
\[
(\mT^{H} * \mA * \mT)\#(\mT^{H} * \mB * \mT)=\mT^{H}*(\mA\#\mB)*\mT.
\]
\end{proposition}
\begin{proof}
Set $\widetilde{\mA}:=\mT^{H} * \mA * \mT$ and $\widetilde{\mB}:=\mT^{H} * \mB * \mT$.
Since $\mA,\mB\in\mathbb{T}_{+}^{n\times n\times p}$ and $\mT$ is nonsingular, we also have $\widetilde{\mA},\widetilde{\mB}\in\mathbb{T}_{+}^{n\times n\times p}$.
Let $\mX:=\mA\#\mB$, so $\mX * \mA^{-1} * \mX = \mB$.
Define $\mY:=\mT^{H} * \mX * \mT$. Using associativity and $(\mT^{H} * \mA * \mT)^{-1}=\mT^{-1} * \mA^{-1} * \mT^{-H}$,
\[
\mY * \widetilde{\mA}^{-1} * \mY
=\mT^{H} * (\mX * \mA^{-1} * \mX) * \mT
=\mT^{H} * \mB * \mT
=\widetilde{\mB}.
\]
Hence $\mY$ solves $\mZ * \widetilde{\mA}^{-1} * \mZ=\widetilde{\mB}$.
Since $\mX\in\mathbb{T}_{+}^{n\times n\times p}$ and congruence preserves strict accretivity, also $\mY\in\mathbb{T}_{+}^{n\times n\times p}$.
By the uniqueness theorem above, $\mY=\widetilde{\mA}\#\widetilde{\mB}$, proving the claim.
\end{proof}

We recall a standard device for describing unitary invariance via singular values.
A mapping $\Phi:\R^{np}\to\R$ is called a \emph{symmetric gauge function} if it is a norm and is invariant under signed permutations; equivalently, $\Phi(\bm x)=\Phi(\bm P\bm J\bm x)$ for every permutation matrix $\bm P$ and every diagonal sign matrix $\bm J$ with diagonal entries $\pm1$~\citep{mirsky1960symmetric}.
Given such a $\Phi$ and a tensor $\mA\in\C^{n\times n\times p}$, we define the induced (T-product) unitarily invariant norm by
$$
\|\mA\|_{\Phi} := \Phi(\sigma(\mA)).
$$
Conversely, a tensor norm is unitary invariant (i.e., $\|\mU^{H} * \mA * \mV\|=\|\mA\|$ for all unitary
$\mU,\mV$) if and only if it arises from some symmetric gauge function in the above way; cf.~\citep{horn2012matrix,bhatia2013matrix}.

Unitarily invariant norms are closely connected to Ky--Fan type majorization inequalities; see, for instance, \citep[Sec.~III]{bhatia2013matrix}. For tensors $\mX, \mY \in \C^{n \times n \times p}$, we define their arithmetic mean by
\begin{equation}\label{eq:arithmetic mean} 
\mX \nabla \mY := \frac{\mX + \mY}{2}. 
\end{equation}
When $\mX, \mY \in \mathbb{H}^{n\times n\times p}$, the Ky--Fan inequality gives the eigenvalue majorization
\begin{equation}\label{eq:KyFan}
2\lambda(\mX \nabla \mY) = \lambda(\mX + \mY) \prec \lambda^{\downarrow}(\mX) + \lambda^{\downarrow}(\mY).
\end{equation}

\begin{proposition}\label{PropTSVDNorm}
If $\mA \in \C^{m \times n \times p}$ has a T-SVD $\mA=\mU*\mS*\mV^H$, then $\|\mA\|_{F}^{2}=\|\mS\|_{F}^{2}=\frac{1}{p}\sum_{i=1}^{p} \|\bm{S}_{i}\|_{F}^{2}$, where $\bm{S}_{i}$ are the diagonal factors in the SVDs of the Fourier-domain slices of $\mA$ (cf.~Lemma~\ref{LemmaFourierBlockDiag}).
\end{proposition}

Let all diagonal elements of $\diag(\bm{S}_{1},\bm{S}_{2},\ldots,\bm{S}_{p})$ be called the T-singular values of $\mA$.
Then a T-SVD of $\mA \in \C^{m\times n\times p}$ is denoted by
\begin{equation}
    \mA = \bcirc^{-1}\left((\bm{F}_p^H \otimes \bm{I}_m) \cdot \left(\sum_{i=1}^{\min\{m,n\}\cdot p} \sigma_{i}(\mA) \bm{u}_{i} \bm{v}_{i}^{H} \right) \cdot (\bm{F}_p \otimes \bm{I}_n)\right),
\end{equation}
where $\sigma_{i}(\mA)$ is the $i$-th diagonal element of $\diag(\bm{S}_{1}, \bm{S}_{2},\ldots, \bm{S}_{p})$ and $\bm{u}_{i}$ and $\bm{v}_{i}$ are the $i$-th column of $\diag(\bm{U}_{1}, \bm{U}_{2},\ldots, \bm{U}_{p})$ and $\diag(\bm{V}_{1},\bm{V}_{2},\ldots,\bm{V}_{p})$, respectively.
Define the set of best rank-$r$ T-SVD truncations of $\mA$ as
\begin{equation}\label{eq:T_SVD_r_approximation_1}
    \begin{aligned}
        \mathcal{S}_{g}(\mA,r):= \bigl\{ \mA_R ~\big|~ R \subset \{1,2,\ldots,\min\{m,n\}\cdot p\},\ |R| = r, \\
        \sigma_i(\mA)\ge \sigma_j(\mA)\ \forall i\in R,\ \forall j\notin R \bigr\}
    \end{aligned}
\end{equation}
where $\mA_R = \bcirc^{-1}\left((\bm{F}_p^H \otimes \bm{I}_m) \cdot \left(\sum_{i\in R} \sigma_{i}(\mA) \bm{u}_{i} \bm{v}_{i}^{H} \right) \cdot (\bm{F}_p \otimes \bm{I}_n)\right)$.

We will repeatedly compare our phase-based model to the classical T-SVD rank model, so we briefly
recall the standard best rank-$r$ approximation problem.
Fix $\mA\in\C^{m\times n\times p}$, an integer $r\in[0,p\min\{m,n\}]$, and a unitarily invariant norm
$\|\cdot\|$ on $\C^{m\times n\times p}$.
\begin{problem}[Low-Rank Approximation]\label{pm:low_rank_approx}
    Let $\mA \in \mathbb{C}^{m \times n \times p}$ and let $\| \cdot \|$ be a unitarily invariant norm.
    Determine
    \begin{equation}\label{eq:low_rank_approx}
        \widehat{\mE} \in \arg\min_{\mE} \bigl\{ \|\mA - \mE\| \;|\; \rank(\mE) \leq r \bigr\},
    \end{equation}
    where $\rank(\cdot)$ denotes the T-SVD induced tensor rank.
\end{problem}
For $\|\cdot\|=\|\cdot\|_F$, one may restrict attention to T-SVD truncations in $\mathcal{S}_g(\mA,r)$, and the optimal value is attained by keeping the $r$ largest T-singular values.
The next lemma summarizes the corresponding tensor analogue of the Schmidt--Mirsky characterization; cf.~\citep[Chap.~IV]{stewart1990matrix}.
\begin{lemma}\label{thm:Schmidt-Mirsky}
    Let $\mA \in \C^{m\times n \times p}$ and let $\|\cdot \|$ be a unitarily invariant norm on
    $\C^{m\times n \times p}$. For $0\leq r\leq p\cdot\min\{m,n\}$,
    $$
    \min_{\rank(\mE) \leq r} \|\mA - \mE \|
    = \left\| \diag\!\bigl(0,\ldots,0,\sigma_{r+1}(\mA),\ldots,\sigma_{p\min\{m,n\}}(\mA)\bigr) \right\|,
    $$
    and every truncated T-SVD approximation $\mE\in\mathcal{S}_g(\mA,r)$ attains the minimum (i.e.,
    it is obtained by zeroing all but the $r$ largest T-singular values).
\end{lemma}

\begin{proof}
Let $\mA \in \C^{m \times n \times p}$ have the T-SVD decomposition $\mA = \mU * \mS * \mV^H$, where $\mS$ is an $m \times n \times p$ tensor with diagonal frontal slices $\mS^{(k)}$ for $k = 1,2,\ldots, p$, and $\sigma_1(\mA) \geq \sigma_2(\mA) \geq \ldots \geq \sigma_{\min\{m,n\}\cdot p}(\mA) \geq 0$ are the T-singular values of $\mA$.

Given a unitarily invariant norm $\|\cdot\|$, we have
$$
\|\mA\| = \|\mU * \mS * \mV^H\| = \|\mS\|,
$$
where $\mU$ and $\mV$ are unitary tensors. 
This equality holds because unitarily invariant norms are functions of the singular values alone.

For a given $r$, consider the rank-$r$ tensor $\mE_r$ defined by truncating $\mS$ such that the first $r$ singular values are retained, and the remaining singular values are set to zero. 
Specifically, define $\mE_r = \mU * \mS_r * \mV^H$, where $\mS_r$ has diagonal slices containing the first $r$ singular values and zeros elsewhere.
Then $\rank(\mE_r)\le r$, and
$$
\mA - \mE_r = \mU * (\mS - \mS_r) * \mV^H.
$$
The norm of the error is then
$$
\|\mA - \mE_r\| = \|\mS - \mS_r\| = \| \diag(0, \ldots, 0, \sigma_{r+1}(\mA), \ldots, \sigma_{\min\{m,n\}\cdot p}(\mA)) \|.
$$

To prove optimality, we reduce to the matrix problem through $\bcirc$. Since $\rank(\mE) = \rank(\bcirc(\mE))$ for any tensor $\mE$, and the induced norm is determined by the same singular values, we obtain
$$
\min_{\rank(\mE) \leq r} \|\mA - \mE\| = \min_{\rank(\bcirc(\mE)) \leq r} \|\bcirc(\mA) - \bcirc(\mE)\|.
$$
By the classical matrix Schmidt-Mirsky theorem \citep[Theorem 2.5.3]{golub2013matrix}, for any matrix $\bm{B}$ with $\rank(\bm{B}) \leq r$, we have
$$
\sigma_{k}(\bcirc(\mA) - \bm{B}) \geq \sigma_{r+k}(\bcirc(\mA)), \quad k = 1, 2, \ldots
$$
Since the singular values of $\bcirc(\mA)$ coincide with the T-singular values of $\mA$, this yields
$$
\|\mA - \mE\| \geq \| \diag(0, \ldots, 0, \sigma_{r+1}(\mA), \ldots, \sigma_{\min\{m,n\}\cdot p}(\mA)) \|
$$
for any $\mE$ with $\rank(\mE) \leq r$. The truncated T-SVD achieves this lower bound.
\end{proof}

We measure phase dispersion by applying a gauge to the canonical T-phase vector.
Let $\Psi:\R^{np}\to\R$ be a symmetric gauge function and let $\mA\in\C^{n\times n\times p}$ be sectorial.
We define the associated \emph{T-phase gauge} by
\[
\measuredangle_{\Psi}\mA := \Psi\!\bigl(\phi(\mA)\bigr).
\]
For $\mA,\mB\in\mathbb{H}^{n\times n\times p}_{++}$, we will also use their (T-product) geometric mean $\mA\#\mB$, which may be characterized as the unique T-positive definite solution $\mX$ of the Riccati-type equation
\begin{equation}\label{eqRiccati}
\mX * \mA^{-1} * \mX = \mB.
\end{equation}

\begin{theorem}\label{thm:geometric_mean1}
For $\mA,\mB \in \mathbb{T}^{n\times n\times p}_{+}$, one has
\[
\phi(\mA \# \mB) \prec \tfrac12\phi(\mA) + \tfrac12\phi(\mB).
\]
\end{theorem}

\begin{proof}
Let $\bm M:=\bcirc(\mA)$ and $\bm N:=\bcirc(\mB)$. 
Since $\mA,\mB\in\mathbb{T}^{n\times n\times p}_{+}$, we have $\bm M,\bm N\in\mathcal{C}(-\tfrac{\pi}{2},\tfrac{\pi}{2})$.

By Proposition~\ref{PropBcircGeometricMean},
\[
\bcirc(\mA\#\mB) = \bcirc(\mA)\#\bcirc(\mB) = \bm M\#\bm N.
\]

By \citep[Theorem~1]{zhao2022low}, for sectorial matrices $\bm M,\bm N\in\mathcal{C}(-\tfrac{\pi}{2},\tfrac{\pi}{2})$,
\[
\phi(\bm M\#\bm N)\prec \tfrac12\phi(\bm M)+\tfrac12\phi(\bm N).
\]

By Definition~\ref{def:canonical_phases}, $\phi(\mZ)=\phi(\bcirc(\mZ))$ for any sectorial tensor $\mZ$.
Hence
\[
\phi(\mA\#\mB)=\phi(\bcirc(\mA\#\mB))=\phi(\bm M\#\bm N)\prec \tfrac12\phi(\bm M)+\tfrac12\phi(\bm N)
=\tfrac12\phi(\mA)+\tfrac12\phi(\mB).
\]
\end{proof}

\begin{corollary}
Let $\mA,\mB \in \mathbb{T}^{n\times n\times p}_{+}$. Then for any symmetric gauge function $\Psi:\R^{np}\rightarrow\R$,
\begin{equation}
    \measuredangle_{\Psi}(\mA \# \mB) \leq \frac{1}{2} \measuredangle_{\Psi} \mA + \frac{1}{2} \measuredangle_{\Psi} \mB.
\end{equation}
\end{corollary}
\begin{proof}
Fan's dominance theorem \citep[II, Theorem~3.17]{stewart1990matrix} implies that $\bm{y}\prec\bm{x}$ entails $\Psi(\bm{y})\le \Psi(\bm{x})$ for every symmetric gauge function $\Psi$.
Combining this characterization with Theorem~\ref{thm:geometric_mean1} and using the convexity of $\Psi$ yields
$$
\measuredangle_{\Psi}(\mA \# \mB) = \Psi(\phi(\mA \# \mB)) \leq
\Psi\left(\frac{\phi(\mA)+\phi(\mB)}{2}\right) \leq \frac{1}{2} \Psi(\phi(\mA))+
\frac{1}{2}\Psi(\phi(\mB)).
$$
\end{proof}

\begin{corollary}
Given $\mA\in\mathbb{T}^{n\times n\times p}_{+}$, it holds
$$
2\phi(\mA^{\frac{1}{2}}) \prec \phi(\mA).
$$
\end{corollary}
\begin{proof}
The identity tensor $\mI_{np}$ belongs to $\mathbb{T}^{n\times n\times p}_{+}$ and satisfies $\phi(\mI_{np}) = \bm{0}$.
We verify that $\mA \# \mI_{np} = \mA^{\frac{1}{2}}$. By the definition of the geometric mean,
\begin{align*}
\mA \# \mI_{np} 
&= \mA^{\frac{1}{2}} * \left(\mA^{-\frac{1}{2}} * \mI_{np} * \mA^{-\frac{1}{2}}\right)^{\frac{1}{2}} * \mA^{\frac{1}{2}} \\
&= \mA^{\frac{1}{2}} * \left(\mA^{-1}\right)^{\frac{1}{2}} * \mA^{\frac{1}{2}} \\
&= \mA^{\frac{1}{2}}.
\end{align*}
By Theorem~\ref{thm:geometric_mean1},
$$
\phi(\mA^{\frac{1}{2}}) = \phi(\mA \# \mI_{np}) \prec \frac{1}{2}\phi(\mA) + \frac{1}{2}\phi(\mI_{np}) = \frac{1}{2}\phi(\mA).
$$
\end{proof}

\begin{corollary}
Let $\mA,\mB\in\mathbb{T}^{n\times n\times p}_{+}$ and consider the family
\begin{equation}\label{eq:set1}
    \left\{ \phi\left( (\mX^{H} *\mA*\mX)\# \mB \right) \Big| \mX \in \C^{n\times n\times p} \text{ is nonsingular} \right\}.
\end{equation}
With respect to the majorization order, this set has a maximal element, namely $(\phi(\mA) + \phi(\mB))/2$.
\end{corollary}
\begin{proof}
For any nonsingular $\mX$, Theorem~\ref{thm:geometric_mean1} implies
$$
\phi\left( (\mX^{H} *\mA*\mX)\# \mB \right) \prec \frac{1}{2} \phi (\mX^{H} *\mA*\mX) + \frac{1}{2} \phi(\mB)
= \frac{1}{2} \phi(\mA) + \frac{1}{2} \phi(\mB),
$$
where we used the congruence invariance $\phi(\mX^H * \mA * \mX)=\phi(\mA)$ for sectorial tensors.
so $(\phi(\mA)+\phi(\mB))/2$ is an upper bound in the majorization sense.
To see that this bound can be attained, let $\mA = \mT^{H}_{\mA} * \mD_{\mA} * \mT_{\mA}$ and $\mB = \mT^{H}_{\mB} * \mD_{\mB} * \mT_{\mB}$ be sectorial decompositions of $\mA$ and $\mB$, respectively.
It follows that

\begin{align*}
\bcirc(\mA)&=(\bm{F}_p^H \otimes \bm{I}_n)\cdot \diag(\bm{A}_1,\bm{A}_2,\ldots,\bm{A}_p)\cdot (\bm{F}_p \otimes \bm{I}_n)\\
&=(\bm{F}_p^H \otimes \bm{I}_n)\cdot \diag(\bm{T}^{H}_{\bm{A}_1},\ldots,\bm{T^{H}}_{\bm{A}_p})\cdot \diag(\bm{D}_{\bm{A}_1},\ldots,\bm{D}_{\bm{A}_p})\cdot \diag(\bm{T}_{\bm{A}_1},\ldots,\bm{T}_{\bm{A}_p})\cdot (\bm{F}_p \otimes \bm{I}_n),\\
\bcirc(\mB)&=(\bm{F}_p^H \otimes \bm{I}_n)\cdot \diag(\bm{B}_1,\bm{B}_2,\ldots,\bm{B}_p)\cdot (\bm{F}_p \otimes \bm{I}_n)\\
&=(\bm{F}_p^H \otimes \bm{I}_n)\cdot \diag(\bm{T}^{H}_{\bm{B}_1},\ldots,\bm{T^{H}}_{\bm{B}_p})\cdot \diag(\bm{D}_{\bm{B}_1},\ldots,\bm{D}_{\bm{B}_p})\cdot \diag(\bm{T}_{\bm{B}_1},\ldots,\bm{T}_{\bm{B}_p})\cdot (\bm{F}_p \otimes \bm{I}_n),\\
\end{align*}
and
\begin{align*}
\bcirc(\mT_{\mA})&=(\bm{F}_p^H \otimes \bm{I}_n)\cdot \diag(\bm{T}_{\bm{A}_1},\bm{T}_{\bm{A}_2},\ldots,\bm{T}_{\bm{A}_p})\cdot (\bm{F}_p \otimes \bm{I}_n),\\
\bcirc(\mD_{\mA})&=(\bm{F}_p^H \otimes \bm{I}_n)\cdot \diag(\bm{D}_{\bm{A}_1},\bm{D}_{\bm{A}_2},\ldots,\bm{D}_{\bm{A}_p})\cdot (\bm{F}_p \otimes \bm{I}_n),\\
\bcirc(\mT_{\mB})&=(\bm{F}_p^H \otimes \bm{I}_n)\cdot \diag(\bm{T}_{\bm{B}_1},\bm{T}_{\bm{B}_2},\ldots,\bm{T}_{\bm{B}_p})\cdot (\bm{F}_p \otimes \bm{I}_n),\\
\bcirc(\mD_{\mB})&=(\bm{F}_p^H \otimes \bm{I}_n)\cdot \diag(\bm{D}_{\bm{B}_1},\bm{D}_{\bm{B}_2},\ldots,\bm{D}_{\bm{B}_p})\cdot (\bm{F}_p \otimes \bm{I}_n),\\
\end{align*}
where each $\bm{A}_{i} = \bm{T}^{H}_{\bm{A}_{i}}\bm{D}_{\bm{A}_{i}}\bm{T}_{\bm{A}_{i}}$ and $\bm{B}_{i} = \bm{T}^{H}_{\bm{B}_{i}}\bm{D}_{\bm{B}_{i}}\bm{T}_{\bm{B}_{i}}$ are matrix sectorial decompositions for $i=1,2,\ldots,p$.
Let $\mX = \mT^{-1}_{\mA} * \mT_{\mB}$.
Then
\begin{align*}
\phi\left( (\mX^{H} * \mA * \mX) \# \mB \right)
=&\phi\left( (\mT_{\mB}^{H} * \mD_{\mA} * \mT_{\mB}) \# (\mT_{\mB}^{H} * \mD_{\mB} * \mT_{\mB}) \right) \\
=&\phi\left( \mT_{\mB}^{H} * ( \mD_{\mA} \# \mD_{\mB}) * \mT_{\mB} \right)\qquad\text{(Proposition~\ref{prop:congruence_geomean})}\\
=&\phi\left( \mD_{\mA} \# \mD_{\mB} \right)\\
=&\frac{1}{2} \phi(\mD_{\mA}) + \frac{1}{2}\phi(\mD_{\mB})\\
=&\frac{1}{2} \phi(\mA) + \frac{1}{2}\phi(\mB).
\end{align*}
\end{proof}

\section{Low T-phase-rank approximation via geometric means}

We now formalize a phase-based notion of rank for sectorial tensors and the associated approximation task.
This is the T-product analogue of the matrix phase-rank concept of~\citep{zhao2022low}: the rank is defined by counting nonzero entries in the canonical phase vector (rather than nonzero singular values).
\begin{definition}
    Let $\mA\in\C^{n\times n\times p}$ be sectorial and let $\phi(\mA)=(\phi_{1}(\mA),\ldots,\phi_{np}(\mA))$ denote the canonical phases in Definition~\ref{def:canonical_phases}.
    The \emph{T-phase rank} is defined by
    \[
        \Tprank(\mA)
        := \bigl|\{\, i\in\{1,\ldots,np\} : \phi_i(\mA)\neq 0 \,\}\bigr|.
    \]
    In particular, $0\le \Tprank(\mA)\le np$, and $\Tprank(\mA)=np$ whenever $0\notin[\underline{\phi}(\mA),\overline{\phi}(\mA)]$.
\end{definition}

In this section, when we state objective-value bounds we work in the positive-imaginary regime:
\[
\mA\in\mathcal{C}[0,\pi),\qquad \mE^{-1}*\mA*\mE^{-1}\in\mathcal{C}[0,\pi),
\]
and all phase vectors are understood as canonical vectors in nonincreasing order.

A convenient bridge between the classical low-rank model and our phase-based formulation is obtained by introducing an auxiliary variable and the arithmetic mean $\nabla$ in~\eqref{eq:arithmetic mean}.
Let $\mX:=\mA-\mE$, so that $\mE=\mA-\mX$.
Noting that $\mA-\mX = 2\bigl((-\mX)\nabla \mA\bigr)$, Problem~\ref{pm:low_rank_approx} can be written as
\begin{equation}\label{eq:low_rank_approximation_arithmetic_mean_1}
    \widehat{\mX}
    = \argmin_{\mX}\Bigl\{ \Phi(\sigma(\mX)) \; \Big|\; \rank\!\bigl(2((-\mX)\nabla \mA)\bigr)\le r\Bigr\}.
\end{equation}
Eliminating $\mE$ yields the equivalent single-variable form
\begin{equation}\label{eq:low_rank_approximation_arithmetic_mean_2}
    \widehat{\mX}
    = \argmin_{\mX}\Bigl\{ \Phi(\sigma(\mX)) \;\Big|\; \rank\bigl((-\mX)\nabla \mA\bigr)\le r \Bigr\}.
\end{equation}
Any optimizer $\widehat{\mX}$ recovers an optimal rank-$r$ approximation via $\widehat{\mE}=\mA-\widehat{\mX}$.

The reformulation above highlights how singular-value majorization enters the classical low-rank approximation problem.
On the phase side, Theorem~\ref{thm:geometric_mean1} plays an analogous role: it provides a majorization principle for canonical phase vectors under the geometric mean.
This motivates a phase-based approximation model in which error is measured through canonical phases and geometric means rather than singular values and arithmetic means.
Accordingly, given a symmetric gauge function $\Psi:\R^{np}\to\R$, a sectorial tensor $\mA\in\C^{n\times n\times p}$, and $0\le r\le np$, we consider
\begin{equation}\label{eq:low_Tprank_approximation_1}
    \widehat{\mX} = \arg\min_{\mX} \left\{ \measuredangle_{\Psi}\mX \;\Big|\; \mX^{-1}, \mA \in \mathbb{T}^{n\times n\times p}_{+},\ \mX^{-1}\#\mA \in \mathcal{C}[0,\pi),\ \Tprank(\mX^{-1} \# \mA) \leq r \right\}.
\end{equation}
Introducing the auxiliary variable $\mE := \mX^{-1} \# \mA$ gives $\mX = \mE^{-1} * \mA * \mE^{-1}$.
Motivated by this identity, we analyze the geometric-mean constrained formulation below.

\begin{problem}\label{pm:low_Tprank_approximation_1}
    Given a symmetric gauge function $\Psi:\R^{np}\to\R$, a sectorial tensor $\mA\in\C^{n\times n\times p}$,
    and an integer $0\le r\le np$, determine
    \begin{align*}
        \min_{\mE\in\C^{n\times n\times p}}\quad
        & \measuredangle_{\Psi}\bigl(\mE^{-1} * \mA * \mE^{-1}\bigr)\\
        \text{s.t.}\quad
        & \mE \ \text{is sectorial},\qquad \mE^{-1} * \mA * \mE^{-1}\ \text{is sectorial},\qquad \Tprank(\mE)\le r.
    \end{align*}
\end{problem}
The sectorial constraint on $\mE$ guarantees $0\notin\sigma(\bcirc(\mE))$; in particular, $\mE$ is invertible and the T-product inverse $\mE^{-1}$ is well defined.

Similar to \eqref{eq:T_SVD_r_approximation_1}, consider the sectorial decomposition \eqref{eq:sectorial_decomposition} for a sectorial tensor $\mA\in\C^{n\times n\times p}$.
By Lemma~\ref{LemmaFourierBlockDiag}, there exist slice-wise sectorial decompositions $\bm{A}_i = \bm{T}_i^{H}\bm{D}_i\bm{T}_i$ for $i=1,\ldots,p$, where each $\bm{T}_i$ is nonsingular and $\bm{D}_i$ is diagonal unitary.
Collecting these, we write
\[
\bcirc(\mA)
=(\bm{F}_p^H \otimes \bm{I}_n)\,
\diag(\bm{T}_{1}^{H},\ldots,\bm{T}_{p}^{H})\,
\diag(\bm{D}_{1},\ldots,\bm{D}_{p})\,
\diag(\bm{T}_{1},\ldots,\bm{T}_{p})\,
(\bm{F}_p \otimes \bm{I}_n).
\]
After a suitable permutation $\bm{P}$ of the diagonal entries to achieve global ordering, we have
\[
\bm{P}^T\diag(\bm{D}_{1},\ldots,\bm{D}_{p})\bm{P}
=\operatorname{diag}\!\big(e^{i\phi_{1}(\mA)},\ldots,e^{i\phi_{np}(\mA)}\big),
\]
where $\phi_1(\mA)\ge\cdots\ge\phi_{np}(\mA)$ are the canonical phases.
Let $\bm{t}_{i}$ denote the $i$-th column of the permuted matrix $\bm{P}^T\diag(\bm{T}_{1},\ldots,\bm{T}_{p})$.

Let $\mA\in\mathbb{C}^{n\times n\times p}$ be sectorial with a sectorial decomposition
\[
\bcirc(\mA)=\bm{T}^{*}\bm{D}\bm{T},
\qquad
\bm{D}=\operatorname{diag}\!\big(e^{i\phi_1(\mA)},\ldots,e^{i\phi_{np}(\mA)}\big),
\]
where the phases are ordered as $\phi_1(\mA)\ge\cdots\ge\phi_{np}(\mA)$.
For $0\le r\le np$, we define the \emph{half-phase diagonal matrix}
\[
\bm{D}_{r/2}(\mA):=\operatorname{diag}\!\big(e^{i\phi_1(\mA)/2},\ldots,e^{i\phi_r(\mA)/2},1,\ldots,1\big)\in\mathbb{C}^{np\times np},
\]
and the corresponding \emph{half-phase truncation tensor}
\[
\mE_{r/2}:=\bcirc^{-1}\!\big(\bm{T}^{*}\bm{D}_{r/2}(\mA)\bm{T}\big).
\]
The set of all $r$-half truncations is then given by
\begin{equation}
\mathcal{S}_{Tp}(\mA,r)
:=\Big\{\mE_{r/2} \;\Big|\; \bcirc(\mA)=\bm{T}^{*}\bm{D}\bm{T} \text{ is a sectorial decomposition with ordered phases}\Big\}.
\end{equation}

We will frequently use the Ky-Fan family of symmetric gauge functions~\citep{fan1955some}. For $x\in\mathbb{R}^{np}$ and $k=1,2,\ldots,np$, define
\begin{equation}
\Gamma_k(x) := \text{sum of the largest } k \text{ terms in } \{|x_1|, |x_2|, \dots, |x_{np}|\},
\quad k=1,2,\ldots,np,
\end{equation}
for $x\in\mathbb{R}^{np}$.
We now collect several basic properties of sectorial tensors for later use.

\begin{lemma}
Let $\mA, \mB \in \mathcal{C}[\alpha, \beta]$ with $0 < \beta - \alpha < \pi$. Then the following holds.
\begin{itemize}
    \item[(a)] $\mA + \mB \in \mathcal{C}[\alpha, \beta]$.
    \item[(b)] If, in addition, $\angle\lambda(\mA * \mB)$ takes values in
    \[
        \big(\gamma(\mA)+\gamma(\mB)-\pi,\ \gamma(\mA)+\gamma(\mB)+\pi\big),
    \]
    where $\gamma(\mX):=(\overline{\phi}(\mX)+\underline{\phi}(\mX))/2$, then
    \[
        \angle\lambda(\mA * \mB) \prec \phi(\mA) + \phi(\mB).
    \]
\end{itemize}
\end{lemma}
\begin{proof}
By Definition~\ref{def:canonical_phases}, tensor phases are identified with matrix phases of the block-circulant realization.
Hence statement (a) follows from \citep[Theorem~3.6]{ding2025tensor}, and statement (b) follows from \citep[Theorem~3.5]{ding2025tensor}.
\end{proof}

\begin{theorem}
Fix a symmetric gauge function $\Psi:\R^{np} \rightarrow \R$ and a tensor $\mA\in\mathcal{C}[0,\pi)$. Let $0\leq r \leq np$ and $\widehat{\mE}\in\mathcal{S}_{Tp}(\mA,r)$. Then
\[
\measuredangle_{\Psi}(\widehat{\mE}^{-1} * \mA * \widehat{\mE}^{-1})
=\Psi (\phi_{r+1}(\mA),\ldots , \phi_{np}(\mA),0,\ldots,0).
\]
Consequently,
\begin{equation}\label{eq:low_Tprank_approximation_3}
\begin{aligned}
\min_{\mE} \Bigl\{& \measuredangle_{\Psi} (\mE^{-1} * \mA * \mE^{-1}) \;\Big|\;
\mE^{-1} * \mA * \mE^{-1}\in \mathcal{C}[0,\pi), \mE \text{ sectorial},\ \Tprank(\mE) \leq r \Bigr\} \\
&= \Psi (\phi_{r+1}(\mA),\ldots , \phi_{np}(\mA),0,\ldots,0).
\end{aligned}
\end{equation}
\end{theorem}
\begin{proof}
Let $\mA$ admit a sectorial decomposition
$$
\mA = \mT^{H} * \mD * \mT,
$$
where $\bcirc(\mT)$ is nonsingular and $\bcirc(\mD) = (\bm{F}_p^H \otimes \bm{I}_n) \cdot \breve{\bm{D}} \cdot (\bm{F}_p \otimes \bm{I}_n)$.
Choose a permutation matrix $\bm{P}$ so that
$$
\bm{P}^T \breve{\bm{D}} \bm{P} = \diag (e^{j\phi_{1}(\mA)},e^{j\phi_{2}(\mA)},\ldots,e^{j\phi_{np}(\mA)}),
$$
with $\phi_{k}(\mA)$ arranged in nonincreasing order.
For $\widehat{\mE} = \mT^{H} * \Lambda * \mT \in \mathcal{S}_{Tp}(\mA,r)$, we have
$$
\bcirc(\Lambda) = (\bm{F}_p^H \otimes \bm{I}_n) \cdot \bm{P} \cdot \diag (e^{j\phi_{1}(\mA)/2},e^{j\phi_{2}(\mA)/2},\ldots,e^{j\phi_{r}(\mA)/2},1,\ldots,1) \cdot \bm{P}^T \cdot (\bm{F}_p \otimes \bm{I}_n).
$$
Using $\mA=\mT^H*\mD*\mT$ and $\widehat{\mE}^{-1}=\mT^{-1}*\Lambda^{-1}*\mT^{-H}$, we obtain
$$
\widehat{\mE}^{-1} * \mA * \widehat{\mE}^{-1}
= \mT^{-1} * (\Lambda^{-1} * \mD * \Lambda^{-1}) * \mT^{-H}
= \mT^{-1} * \mathrm{V} * \mT^{-H},
$$
where
$$
\bcirc(\mathrm{V}) = (\bm{F}_p^H \otimes \bm{I}_n) \cdot \diag(1, \ldots, 1, e^{j\phi_{r+1}(\mA)}, e^{j\phi_{r+2}(\mA)}, \ldots, e^{j\phi_{np}(\mA)}) \cdot (\bm{F}_p \otimes \bm{I}_n).
$$
Hence
$$
\measuredangle_{\Psi}(\widehat{\mE}^{-1} * \mA * \widehat{\mE}^{-1}) = \Psi (\phi_{r+1}(\mA),\ldots , \phi_{np}(\mA),0,\ldots,0),
$$
since the canonical phase vector is ordered nonincreasingly and $\Psi$ is permutation invariant.
Thus the right-hand side of \eqref{eq:low_Tprank_approximation_3} is attainable.

For the reverse inequality, let $\mE$ be any feasible tensor in \eqref{eq:low_Tprank_approximation_3}, and set
\[
\bm{A}:=\bcirc(\mA),\qquad \bm{E}:=\bcirc(\mE).
\]
Then $\bm{A}$ is sectorial with canonical phases in $[0,\pi)$, $\bm{E}$ is sectorial, $\bm{E}^{-1}\bm{A}\bm{E}^{-1}$ is sectorial with canonical phases in $[0,\pi)$, and
\[
\prank(\bm{E})=\Tprank(\mE)\le r.
\]
These are exactly the assumptions of \citep[Theorem~2]{zhao2022low} (dimension $np$), so
\[
\Psi\!\bigl(\phi(\bm{E}^{-1}\bm{A}\bm{E}^{-1})\bigr)
\ge
\Psi (0,\ldots,0,\phi_{r+1}(\bm{A}),\ldots,\phi_{np}(\bm{A})).
\]
By permutation invariance of $\Psi$, the right-hand side equals
\[
\Psi(\phi_{r+1}(\bm A),\ldots,\phi_{np}(\bm A),0,\ldots,0).
\]
Using $\phi(\mZ)=\phi(\bcirc(\mZ))$ and $\bcirc(\mE^{-1}*\mA*\mE^{-1})=\bm{E}^{-1}\bm{A}\bm{E}^{-1}$ gives
\[
\measuredangle_{\Psi}(\mE^{-1} * \mA * \mE^{-1})
\ge
\Psi (\phi_{r+1}(\mA),\ldots,\phi_{np}(\mA),0,\ldots,0).
\]
Taking the minimum over all feasible $\mE$ gives the lower bound; combined with attainability, this proves \eqref{eq:low_Tprank_approximation_3}.
\end{proof}

Theorem~\ref{thm:geometric_mean1} and \eqref{eq:low_Tprank_approximation_3} therefore provide an explicit optimal value formula in the positive-imaginary regime, and every $\widehat{\mE}\in\mathcal{S}_{Tp}(\mA,r)$ is optimal.

A natural question is whether the low-$\Tprank$ approximation can be reformulated as a rank-constrained problem, analogous to classical low-rank approximation.
The next result compares $\Tprank(\cdot)$ with the classical tensor rank and highlights their difference.

\begin{theorem}\label{thm:statements}
Let $\mA \in \C^{n\times n\times p}$ be sectorial. Then:
\begin{itemize}
    \item[(a)] There exists a tensor $\mR\in\C^{n\times n\times p}$ such that $\rank(\mR) = \Tprank(\mA)$ and
    $$
        \mA-\mR \in\ \mathbb H^{n\times n\times p}_{++}.
    $$
    \item[(b)] One always has
    \[
        \Tprank(\mA) \geq \min\limits_{\mM \in\ \mathbb H^{n\times n\times p}_{++}} \rank(\mA - \mM)  = \min_{\substack{\mT\in\C^{n\times n\times p}\\ \det(\bcirc(\mT))\neq 0}} \rank (\mT^{H} * \mA * \mT - \mI_{np}),
    \]
    where $\mI_{np}$ is the identity tensor.
\end{itemize}
\end{theorem}
\begin{proof}
Let $\mA = \mT^{H} * \mD * \mT$ be a sectorial decomposition with nonsingular $\mT$ and diagonal unitary $\mD$. In terms of the block-circulant embedding this means
$$
\bcirc(\mD) = (\bm{F}_p^H \otimes \bm{I}_n) \cdot \diag(e^{j\phi_1(\mA)}, e^{j\phi_2(\mA)}, \ldots, e^{j\phi_{np}(\mA)}) \cdot (\bm{F}_p \otimes \bm{I}_n).
$$

For statement (a), let $\mM = \mT^{H} * \mT$. Since $\mT$ is nonsingular, $\mM$ is T-positive definite, i.e., $\mM \in \mathbb{H}^{n\times n\times p}_{++}$.
Define $\mR = \mA - \mM = \mT^{H} * (\mD - \mI_{np}) * \mT$.

We now show that $\rank(\mR) = \Tprank(\mA)$. By the properties of the T-product and the DFT block-diagonalization,
$$
\bcirc(\mD - \mI_{np}) = (\bm{F}_p^H \otimes \bm{I}_n) \cdot \diag(e^{j\phi_1(\mA)} - 1, \ldots, e^{j\phi_{np}(\mA)} - 1) \cdot (\bm{F}_p \otimes \bm{I}_n).
$$
The rank of $\bcirc(\mD - \mI_{np})$ equals the number of nonzero diagonal entries, which is
$$
\#\{k : e^{j\phi_k(\mA)} \neq 1\} = \#\{k : \phi_k(\mA) \neq 0\} = \Tprank(\mA).
$$
Since $\bcirc(\mR) = \bcirc(\mT)^H \bcirc(\mD - \mI_{np}) \bcirc(\mT)$ and $\bcirc(\mT)$ is nonsingular, we have
$$
\rank(\mR) = \rank(\bcirc(\mR)) = \rank(\bcirc(\mD - \mI_{np})) = \Tprank(\mA).
$$
Moreover, $\mA - \mR = \mM \in \mathbb{H}^{n\times n\times p}_{++}$.

For statement (b), there exists $\mR$ with $\rank(\mR) = \Tprank(\mA)$ such that $\mA - \mR \in \mathbb{H}^{n\times n\times p}_{++}$ from statement (a).
This shows that $\Tprank(\mA) \geq \min_{\mM \in \mathbb{H}^{n\times n\times p}_{++}} \rank(\mA - \mM)$.
Conversely, for any $\mM \in \mathbb{H}^{n\times n\times p}_{++}$, we can write $\mM = \mT_0^H * \mT_0$ for some nonsingular $\mT_0$.
Then $\mA - \mM = \mT_0^H * (\mT_0^{-H} * \mA * \mT_0^{-1} - \mI_{np}) * \mT_0$, so
$$
\rank(\mA - \mM) = \rank(\mT_0^{-H} * \mA * \mT_0^{-1} - \mI_{np}) = \rank(\mT^H * \mA * \mT - \mI_{np})
$$
where $\mT = \mT_0^{-1}$. 
Conversely, for any nonsingular $\mT$, setting $\mM:=\mT^{-H} * \mT^{-1}\in\mathbb{H}^{n\times n\times p}_{++}$ gives
$$
\rank(\mT^{H} * \mA * \mT - \mI_{np})=\rank(\mA-\mM).
$$
Taking minima over $\mM\in\mathbb{H}^{n\times n\times p}_{++}$ and over nonsingular $\mT$ yields the equality of the two minimum expressions in (b).
\end{proof}

Majorization provides a convenient language for comparing phase vectors and will be used throughout.
In the eigenvalue setting, Ky--Fan type bounds admit a refinement known as Lidskii's inequality; see,
for example,~\citep[Cor.~III.4.2]{kubo1980means}. For $\mX,\mY\in\mathbb{H}^{n\times n\times p}$,
\begin{equation}\label{eq:Lidskii inequality}
    2\lambda^\downarrow(\mX \nabla \mY) - \lambda^\downarrow(\mX)
    = \lambda^\downarrow(\mX + \mY) - \lambda^\downarrow(\mX)
    \prec \lambda^\downarrow(\mY).
\end{equation}
A standard route to Schmidt--Mirsky type results is to apply~\eqref{eq:Lidskii inequality} to Hermitian dilation constructed from $\bcirc(\mE)$ and $\bcirc(\mA-\mE)$, namely
$$
\begin{bmatrix} 0 & \bcirc(\mE) \\ \bcirc(\mE)^{H} & 0 \end{bmatrix},
\qquad
\begin{bmatrix} 0 & \bcirc(\mA-\mE) \\ (\bcirc(\mA-\mE))^{H} & 0 \end{bmatrix}
\in \mathbb{H}^{2np\times 2np}.
$$
Since Theorem~\ref{thm:geometric_mean1} gives a Ky--Fan-like majorization bound for canonical phases under
the geometric mean, one may ask whether a phase-domain analogue of Lidskii's inequality holds, namely
\begin{equation}\label{eq:inequality}
    2\phi(\mA \# \mB) - \phi(\mA) \prec \phi(\mB),
    \qquad \forall\,\mA,\mB\in\mathbb{T}^{n\times n\times p}_{+}.
\end{equation}
Establishing or refuting \eqref{eq:inequality} in full generality is beyond the scope of this paper.

The formulation above can also be written by swapping the roles of $\mE$ and $\mX$.
Indeed, since $\mE=\mA-\mX$ is equivalent to $\mX=\mA-\mE$, we may rewrite the problem as
\begin{equation}\label{eq:low rank approximation}
    \widehat{\mX}
    = \argmin_{\mX}\Bigl\{\Phi(\sigma(\mX)) \;:\; \rank\bigl((-\mX)\nabla \mA\bigr)\le r\Bigr\}.
\end{equation}
This equivalent form highlights the analogy between the arithmetic-mean low-rank model
and our geometric-mean phase model.

Algorithm~\ref{alg:tphase} summarizes the construction of a half-phase truncation tensor
$\mE^*\in\mathcal{S}_{Tp}(\mA,r)$.
The computation is performed slice-wise in the Fourier domain: we compute sectorial decompositions of the
Fourier slices, sort the canonical phases across all slices, and then assemble a diagonal half-phase factor so that
$(\mE^*)^{-1}*\mA*(\mE^*)^{-1}$ retains the $np-r$ smallest phases, equivalently, the $r$ largest phases are canceled.
For reproducibility, phase sorting is deterministic (ties are broken by slice index, then in-slice index).
With dense linear algebra, the dominant cost is sectorial decomposition of $p$ frontal slices, giving
$\mathcal{O}(pn^3)$ flops, while the global phase sort costs $\mathcal{O}(np\log(np))$.
All reported phase values are obtained from principal arguments of Fourier-slice eigenvalues using this deterministic sorting rule.

\begin{algorithm}[t!]
\DontPrintSemicolon
\caption{Low T-Phase-Rank Approximation \label{alg:tphase}}
\KwIn{$\mA\in\mathbb{C}^{n\times n\times p}$ (sectorial), $r\in\mathbb{N}$}
\KwOut{$\mE^*$}
$\bcirc(\mA)=(\bm F_p^{H}\otimes \bm I_n) \diag(\bm A_1,\ldots,\bm A_p) (\bm F_p\otimes \bm I_n)$\;
\For{$i=1,\ldots,p$}{
$\bm A_i=\bm T_i^{H}\bm D_i\bm T_i,\quad \bm D_i=\diag\big(e^{ j\theta_{i,1}},\ldots,e^{j\theta_{i,n}}\big)$\;
}
$\widehat{\bm T}_{\text{blk}}=\diag(\bm T_1,\ldots,\bm T_p)$\;
$\widehat{\bm D}_{\text{blk}}=\diag(\bm D_1,\ldots,\bm D_p)$\;
Find permutation matrix $\bm P$ such that $\bm P^T \widehat{\bm D}_{\text{blk}} \bm P = \diag(e^{j\phi_1(\mA)},\ldots,e^{j\phi_{np}(\mA)})$ with sorted phases $\phi_1(\mA) \ge \ldots \ge \phi_{np}(\mA)$\;
$\widehat{\bm \Lambda}^{(r)}_{\text{sorted}}=\diag\big(e^{j\phi_1(\mA)/2},\ldots,e^{j\phi_r(\mA)/2},\underbrace{1,\ldots,1}_{np-r}\big)$\;
$\widehat{\bm \Lambda}^{(r)}_{\text{blk}} = \bm P \widehat{\bm \Lambda}^{(r)}_{\text{sorted}} \bm P^T$\;
$\widehat{\bm E}_{\text{blk}} = \widehat{\bm T}_{\text{blk}}^H \widehat{\bm \Lambda}^{(r)}_{\text{blk}} \widehat{\bm T}_{\text{blk}}$\;
$\mE^* = \bcirc^{-1}\big((\bm F_p^{H}\otimes \bm I_n) \widehat{\bm E}_{\text{blk}} (\bm F_p\otimes \bm I_n)\big)$\;
\Return $\mE^*$\;
\end{algorithm}

\begin{example}
Let $n=2$ and $p=3$.
In the frequency domain, set the frontal slices
$$
\bm A_1=\diag\big(e^{0.6j},e^{0.2j}\big),\qquad
\bm A_2=\diag\big(e^{0.4j},e^{0.1j}\big),\qquad
\bm A_3=\diag\big(e^{0.3j},e^{0.05j}\big),
$$
and define the tensor $\mA\in\C^{2\times2\times3}$ by
$$
\bcirc(\mA)=\big(\bm F_{3}^{\mathrm H}\otimes \bm I_{2}\big)
\diag(\bm A_1,\bm A_2,\bm A_3)
\big(\bm F_{3}\otimes \bm I_{2}\big).
$$
Then $\mA\in\mathcal{C}[0,\pi)$ and the canonical phase vector is
$$
\phi(\mA) = (0.6,0.4,0.3,0.2,0.1,0.05).
$$
Fix $r=3$ and apply Algorithm~\ref{alg:tphase}.
Since $\bm A_k$ ($k=1,2,3$) are already diagonal unitary, the sectorial decomposition yields $\bm T_k = \bm I_2$ for all $k$.
The sorted half-phase diagonal matrix is
$$
\widehat{\bm \Lambda}^{(r)}_{\mathrm{sorted}}=\diag\big(e^{0.3j},e^{0.2j},e^{0.15j},1,1,1\big),
$$
and after permuting back to block structure,
$$
\widehat{\bm \Lambda}^{(r)}_{\mathrm{blk}}=\diag\Big(\diag(e^{0.3j},1), \diag(e^{0.2j},1), \diag(e^{0.15j},1)\Big).
$$
Hence the output tensor satisfies $\bcirc(\mE^*)=(\bm F_3^{\mathrm H}\otimes\bm I_2)\widehat{\bm \Lambda}^{(r)}_{\mathrm{blk}}(\bm F_3\otimes\bm I_2)$.

Let $\mW:=(\mE^*)^{-1}*\mA*(\mE^*)^{-1}$. Then the frequency-domain slices of $\mW$ are
$$
\bm W_1=\diag(1,e^{0.2j}),\qquad
\bm W_2=\diag(1,e^{0.1j}),\qquad
\bm W_3=\diag(1,e^{0.05j}),
$$
and 
$$
\phi(\mW)=(0.2,0.1,0.05,0,0,0)\quad\text{and}\quad \Tprank(\mE^*)=3.
$$
Thus, the top $r$ phases are driven to $0$ while the remaining phases are preserved.
\end{example}

\section{Phase of tensor MIMO LTI systems and small phase theorem}

Recent work has explored representing MIMO (multi-input multi-output) LTI dynamics using third-order tensors together with the T-product; see, e.g., \citep{che2022fast,che2024randomized,wang2025algebraic}.

For $\mathcal{A},\mathcal{B},\mathcal{C},\mathcal{D} \in \mathbb{R}^{m\times m\times p}$ we consider the continuous-time tensor state-space model
\begin{equation}\label{eq:tensor_mimo}
    \begin{cases}
        \dot{\mathcal{X}}(t) =  \mathcal{A}*\mathcal{X}(t) + \mathcal{B}*\mathcal{U}(t)\\
        \mathcal{Y}(t) = \mathcal{C}*\mathcal{X}(t) + \mathcal{D}*\mathcal{U}(t)
    \end{cases}
\end{equation}
where $\mathcal{X}(t)\in \mathbb{R}^{m\times 1 \times p}$ is the state tensor, $\mathcal{U}(t)\in \mathbb{R}^{m\times 1 \times p}$ is the input (control) tensor, and $\mathcal{Y}(t)\in \mathbb{R}^{m\times 1 \times p}$ is the output tensor.
The associated transfer tensor is
\begin{equation}
    \mG(s) = \mC * (s\mI - \mA)^{-1} * \mB + \mD,
\end{equation}
whose entries are real-rational proper functions of $s$.
Let $\mathcal{RH}_{\infty}^{m \times m}$ denote the set of real-rational proper stable transfer matrices, and let $\mathcal{RH}_{\infty}^{m \times m \times p}$ be the corresponding tensor class (equivalently, $\mG\in\mathcal{RH}_{\infty}^{m \times m \times p}$ iff $\bcirc(\mG)\in \mathcal{RH}_{\infty}^{mp \times mp}$).

For $\mG \in \mathcal{RH}_{\infty}^{m \times m \times p}$, the frequency-dependent vector $\sigma(\mG(j\omega))\in\R^{mp}$ describes the magnitude response via T-singular values.
The associated $\mathcal{H}_{\infty}$ norm is
\begin{equation*}
    \| \mG \|_{\infty} = \sup_{\omega \in \R} \overline{\sigma}(\mG (j\omega)).
\end{equation*}

\begin{example}
    Consider a tensor LTI system $\mathcal{G}(s) \in \mathbb{C}^{2 \times 2 \times 2}$ specified by two frontal slices $G^{(1)}(s)$ and $G^{(2)}(s)$:
    \begin{align*}
        \bm{G}^{(1)}(s) =&~ \frac{1}{s^2 + 2s + 2}\begin{bmatrix}
            2s^2 + 3s + 4 & 0.5s^2 + s + 1 \\
            0.5s^2 + s + 1 & 1.5s^2 + 2s + 3
        \end{bmatrix}\\
        \bm{G}^{(2)}(s) =&~ \frac{1}{s^2 + 2s + 2}\begin{bmatrix}
            0.8s^2 + s + 1 & 0.2s^2 + 0.5s + 0.2 \\
            0.2s^2 + 0.5s + 0.2 & 0.9s^2 + 1.2s + 1
        \end{bmatrix}
    \end{align*}
    Figure~\ref{fig:bodeplot_sectorial} shows the Bode plot of the block-circulant realization.
    \begin{figure}[h]
        \centering
        \includegraphics[width=0.8\linewidth]{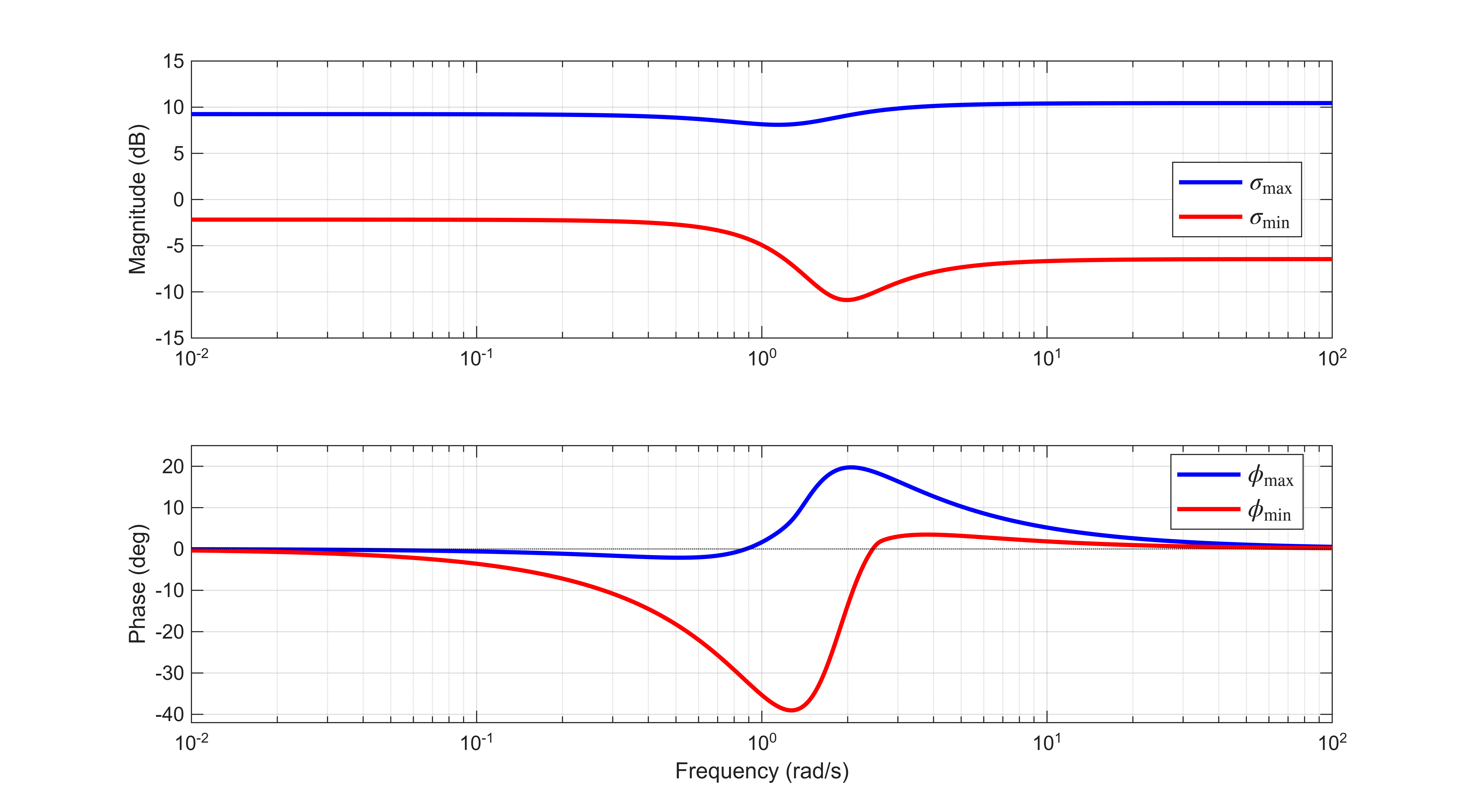}
        \caption{Bode plot of the tensor MIMO system}
        \label{fig:bodeplot_sectorial}
    \end{figure}

    Using a log-spaced frequency grid over the plotted Bode range (the same sweep used for Figure~\ref{fig:bodeplot_sectorial}), we evaluate
    $\underline{\Phi}_\infty(\mathcal{G})$ and $\overline{\Phi}_\infty(\mathcal{G})$ from the canonical phase bounds, obtaining
    \[
    \bigl[\underline{\Phi}_\infty(\mathcal{G}),\,\overline{\Phi}_\infty(\mathcal{G})\bigr]
    \approx [-39.04^\circ,\, 19.74^\circ].
    \]
    Since the phase spread is approximately $58.8^\circ \leq 180^\circ$, the system is sectorial, i.e., $\mathcal{G}\in \mathcal{C}(\alpha,\beta)$ for some $\beta-\alpha\le \pi$ (and the cone $\mathcal{C}(\alpha,\beta)$ is convex).
\end{example}

In the remainder of this section we restrict attention to systems whose frequency responses admit a consistent principal-phase branch.
Concretely, we assume that for every $\omega\in\R$ the numerical range $W(\mG(j\omega))$ avoids the negative real axis; this guarantees that the canonical phase vector is well defined.
We call $\phi(\mG(j\omega))\in\R^{mp}$ the \emph{phase response} of $\mG$.

To summarize this frequency-dependent phase information by a single scalar, we use the $\mathcal{H}_{\infty}$-type phase measure
\begin{equation*}
    \Phi_{\infty}(\mG) = \sup_{\omega \in \R, \|\bm{x}\|=1} \angle \bm{x}^H \bcirc(\mG(j\omega))\bm{x},
\end{equation*}
which plays the role of an $\mathcal{H}_{\infty}$ phase specification.
Clearly $\Phi_{\infty}(\mG) < \pi$. For $\alpha\in[0,\pi)$ we define the phase-bounded class
\begin{equation*}
    \mathsf{C}[\alpha] = \{\mG \in \mathcal{RH}_{\infty}^{m \times m \times p} : \Phi_{\infty}(\mG) < \alpha\}.
\end{equation*}

Suppose $\mG, \mH$ are both $m \times m \times p$ real rational proper stable transfer function tensors. The feedback interconnection of $\mG$ and $\mH$ is said to be stable if the Gang of Four tensor
\begin{equation*}
    \mG \widetilde{\#}\mH =
    \begin{bmatrix}
        (\mI_{mp} + \mH * \mG)^{-1} & (\mI_{mp} + \mH * \mG)^{-1} * \mH \\
        \mG * (\mI_{mp} + \mH * \mG)^{-1} & \mG * (\mI_{mp} + \mH * \mG)^{-1} * \mH
    \end{bmatrix}
\end{equation*}
is stable, i.e., $\mG \widetilde{\#}\mH \in \mathcal{RH}_{\infty}^{2m \times 2m \times p}$.
We seek sufficient and necessary conditions for the stability of the tensor feedback interconnection.

\begin{lemma}\label{lem:feedback_bcirc_equiv}
    The feedback interconnection of $\mG$ and $\mH$ is stable if and only if the feedback interconnection of $\bcirc(\mG)$ and $\bcirc(\mH)$ is stable.
\end{lemma}
\begin{proof}
    By multiplicativity of $\bcirc$ and the inverse identity for the T-product,
    \begin{equation*}
        \bcirc((\mI_{mp} + \mH * \mG)^{-1}) = \bcirc(\mI_{mp} + \mH * \mG)^{-1} = (\bm{I}_{mp} + \bcirc(\mH) \bcirc(\mG))^{-1}.
    \end{equation*}
    Moreover, there exists a permutation matrix $\bm{P}$ that satisfies
    \begin{equation*}
        \bcirc(\mG \widetilde{\#}\mH)=\bm{P}^{T} 
        \begin{bmatrix}
            \bcirc((\mI_{mp} + \mH * \mG)^{-1}) & \bcirc((\mI_{mp} + \mH * \mG)^{-1} * \mH) \\
            \bcirc(\mG * (\mI_{mp} + \mH * \mG)^{-1}) & \bcirc(\mG * (\mI_{mp} + \mH * \mG)^{-1} * \mH)
        \end{bmatrix} \bm{P}
    \end{equation*}
    Combining the above equations and letting $\bm{T} = (\bm{I}_{mp} + \bcirc(\mH) \bcirc(\mG))^{-1}$, we obtain
    \begin{equation*}
        \bcirc(\mG \widetilde{\#}\mH)=\bm{P}^{T} 
        \begin{bmatrix}
            \bm{T} & \bm{T} \bcirc(\mH) \\
            \bcirc(\mG)\bm{T} & \bcirc(\mG) \bm{T} \bcirc(\mH)
        \end{bmatrix} \bm{P} = \bm{P}^{T} \bcirc(\mG )\widetilde{\#} \bcirc(\mH) \bm{P},
    \end{equation*}
    Thus $\bcirc(\mG \widetilde{\#}\mH)$ is permutation-similar to $\bcirc(\mG)\widetilde{\#}\bcirc(\mH)$.
    Since permutation similarity preserves poles and stability, the two feedback interconnections are stable simultaneously.
\end{proof}

The small gain theorem is a basic tool in robust control~\citep{Chen2020small,liu2016robust,khalil1996robust}.
In the matrix setting it asserts that the feedback interconnection is stable when the loop gain satisfies
$\overline{\sigma}(\bm{G}(j\omega))\,\overline{\sigma}(\bm{H}(j\omega))<1$ for all $\omega$.
Through the block-circulant representation, the same frequency-domain argument applies to transfer tensors, yielding the following tensor formulation.

\begin{theorem}\label{thm:theorem 5.4}
    For $\mG, \mH \in \mathcal{RH}_{\infty}^{m \times m \times p}$, the feedback system $\mG \widetilde{\#} \mH$ is stable if
    \begin{equation*}
        \overline{\sigma}(\mG(j\omega)) \overline{\sigma}(\mH(j\omega)) < 1
    \end{equation*}
    holds for all $\omega \in \mathbb{R}$.
\end{theorem}
\begin{proof}
    By T-SVD spectral invariance under $\bcirc$,
    \begin{equation*}
        \overline{\sigma}(\mG(j\omega)) = \overline{\sigma}(\bcirc(\mG(j\omega))), \quad 
        \overline{\sigma}(\mH(j\omega)) = \overline{\sigma}(\bcirc(\mH(j\omega)))
    \end{equation*}
    for all $\omega\in\mathbb{R}$. Hence
    \begin{equation*}
        \overline{\sigma}(\bcirc(\mG(j\omega))) \overline{\sigma}(\bcirc(\mH(j\omega))) < 1
    \end{equation*}
    for all $\omega \in \mathbb{R}$.
    Note that $\bcirc(\mG),\bcirc(\mH) \in \mathcal{RH}_{\infty}^{mp\times mp}$.
    By the matrix small gain theorem, the feedback system $\bcirc(\mG) \widetilde{\#} \bcirc(\mH)$ is stable.
    By Lemma~\ref{lem:feedback_bcirc_equiv}, this is equivalent to stability of $\mG \widetilde{\#} \mH$.
\end{proof}

Besides gain-based criteria, phase bounds also provide stability certificates for feedback interconnections.
The small phase theorem plays a central role in phase-based analysis of MIMO systems.
For $\mG \in \mathcal{RH}_{\infty}^{m \times m \times p}$, the frequency response is conjugate symmetric, i.e., $\mG(-j\omega)=\overline{\mG(j\omega)}$, and hence $W(\mG(j\omega))$ and $W(\mG(-j\omega))$ are symmetric with respect to the real axis.
Using the tensor phase notions developed above, we obtain a tensor analogue of the small phase theorem.

\begin{theorem}
    Let $\mG \in \mathcal{RH}_{\infty}^{m \times m \times p}$ be frequency-wise quasi-sectorial, and let $\mH \in \mathcal{RH}_{\infty}^{m \times m \times p}$ be frequency-wise semi-sectorial. Then the feedback system $\mG \widetilde{\#} \mH$ is stable if
    \begin{equation*}
        \overline{\phi}(\mG(j\omega)) + \overline{\phi}(\mH(j\omega)) < \pi \quad \text{and} \quad \underline{\phi}(\mG(j\omega)) + \underline{\phi}(\mH(j\omega)) > -\pi
    \end{equation*}
    hold for all $\omega \in [0, \infty]$.
\end{theorem}
\begin{proof}
    For $\omega\ge 0$, the assumptions give
    \[
    \overline{\phi}(\mG(j\omega)) + \overline{\phi}(\mH(j\omega)) < \pi,\qquad
    \underline{\phi}(\mG(j\omega)) + \underline{\phi}(\mH(j\omega)) > -\pi.
    \]
    Since $\mG(-j\omega)=\overline{\mG(j\omega)}$ and $\mH(-j\omega)=\overline{\mH(j\omega)}$, canonical phases satisfy
    \[
    \overline{\phi}(\mG(-j\omega))=-\underline{\phi}(\mG(j\omega)),\quad
    \underline{\phi}(\mG(-j\omega))=-\overline{\phi}(\mG(j\omega)),
    \]
    and similarly for $\mH$. Therefore, the above two inequalities on $[0,\infty)$ imply
    \[
    \overline{\phi}(\mG(j\omega)) + \overline{\phi}(\mH(j\omega)) < \pi,\qquad
    \underline{\phi}(\mG(j\omega)) + \underline{\phi}(\mH(j\omega)) > -\pi
    \]
    for all $\omega\in\mathbb{R}$.
    Using $\phi(\mZ)=\phi(\bcirc(\mZ))$, the same phase bounds hold for $\bcirc(\mG)$ and $\bcirc(\mH)$.
    Since $\bcirc(\mG),\bcirc(\mH)\in\mathcal{RH}_{\infty}^{mp\times mp}$, \citep[Theorem 1]{chen2019phase} yields stability of $\bcirc(\mG)\widetilde{\#}\bcirc(\mH)$.
    By the preceding lemma (equivalence under $\bcirc$), $\mG\widetilde{\#}\mH$ is stable.
\end{proof}

The small phase theorem gives a sufficient, but generally not necessary, condition for stability.
One can formulate variants under weaker assumptions (for instance, $\mG$ frequency-wise quasi-sectorial and $\mH$ frequency-wise semi-sectorial), and in certain settings necessary-and-sufficient phase conditions are available; see \citep{chen2024phase2} for a representative result.

\section*{ Declaration of competing interest}
The authors declare that they have no competing interests.

\section*{ Data availability}
No external dataset was used.
The scripts for tensor generation, phase truncation, and experiment logs (including random seeds and environment information) will be released with the final camera-ready version.

\section*{ Acknowledgments}
The authors contributed equally to this paper.

\section*{ Funding}
H.~Choi and T.~Kim were supported by the National Research Foundation of Korea (NRF) grant funded by the Korea government (MSIT) (Nos.~2022R1A5A1033624 and RS-2024-00342939).
T.~Kim was additionally supported by the National Research Foundation of Korea (NRF) grant funded by the Korea government (MSIT) (No.~RS-2025-25436769).
Yimin Wei is supported by the National Natural Science Foundation of China under grant 12271108, the Ministry of Science and Technology of China under grant H20240841 and the Joint Research Project between China and Serbia under the grant 2024-6-7. We would like to thank Prof. Li Qiu and Mr. Chengdong Liu for their useful suggestions.

\bibliographystyle{unsrtnat}
\bibliography{references}

\end{document}